\documentclass[a4paper,11pt,reqno]{article} 
\usepackage{datetime} 
\usepackage{amsmath}
\allowdisplaybreaks[1]
\usepackage{amsthm}
\usepackage{dsfont}
\usepackage{enumerate}
\usepackage{amssymb}
\usepackage{tikz}
\usepackage{appendix}
\usepackage{graphicx}
\usetikzlibrary{positioning}
\usepackage{multicol}
\usepackage{enumitem}
\usepackage{float}
\usepackage{authblk}
\usepackage[colorinlistoftodos]{todonotes}
\usepackage{accents}
\usepackage{hyperref}
\usepackage{soul}
\hypersetup{
pdftitle={Segre products of cluster algebras},
pdfauthor={Jan E. Grabowski and Lauren Hindmarch},
pdfstartview=FitH
}
\tikzset{node distance=2cm, auto}
\usepackage{caption}
\newtheorem{theorem}{Theorem}[section]

\newtheorem{example}[theorem]{Example}
\newtheorem{corollary}[theorem]{Corollary}
\newtheorem{lemma}[theorem]{Lemma}

\newtheorem{remark}[theorem]{Remark}

\newcommand{\bigdsum}{\bigoplus}
\newcommand{\bK}{\mathbb{K}}
\newcommand{\cA}{\mathcal{A}}
\newcommand{\cluster}[2][x]{\underline{#1}_{#2}}
\newcommand{\defeq}{\stackrel{\scriptscriptstyle{\mathrm{def}}}{=}}

\newcommand{\exch}[1]{\underline{\textup{ex}}_{#1}}
\renewcommand{\hat}{\widehat}
\DeclareMathOperator{\Image}{Im}
\newcommand{\inj}{\hookrightarrow}
\newcommand{\integ}{\mathbb{Z}}
\newcommand{\nat}{\mathbb{N}}

\newcommand{\proj}[1]{\mathbb{P}^{#1}}
\newcommand{\tensor}{\otimes}
\newcommand\thickbar[1]{\accentset{\rule{.5em}{.5pt}}{#1}}
\newcommand{\segre}{\thickbar{\tensor}}
\newcommand{\SegreCA}[2]{\cA_{#1}\overline{\tensor} \cA_{#2}}
\newcommand{\union}{\cup}

\title{Segre products of cluster algebras}
\author[1]{Jan E. Grabowski}
\author[2]{Lauren Hindmarch}
\affil[1,2]{School of Mathematical Sciences, Lancaster University, Lancaster, LA1 4YF, United Kingdom}

\date{11th September 2024}

\begin{document}

\maketitle

\footnotetext[1]{Email: \url{j.grabowski@lancaster.ac.uk}.  Website: \url{http://www.maths.lancs.ac.uk/~grabowsj/}}
\footnotetext[2]
{Email: \url{l.hindmarch@lancaster.ac.uk}.}
\renewcommand{\thefootnote}{\arabic{footnote}}
\setcounter{footnote}{0}

\begin{abstract}
    \noindent We show that under mild assumptions the Segre product of two graded cluster algebras has a natural cluster algebra structure.
    
    \begin{description}
    	\item[MSC:] 13F60
    \end{description}
\end{abstract}

\section{Introduction}

The map \(\sigma: \proj{n} \times \proj{m} \inj \proj{n+m+nm}\) of projective spaces defined by 
\[ \sigma((x_0\colon \dotsc \colon x_n), (y_0 \colon \dotsc \colon y_m)) = (x_0y_0\colon x_1\colon y_0\colon  \dotsc \colon x_ny_m) \]
is known as the \textit{Segre embedding}---it is injective and its image is a subvariety of \(\proj{n+m+nm}\). We may then define the Segre product of two projective varieties \(X \subseteq \proj{n}\) and \(Y \subseteq \proj{m}\) as the image of \(X \times Y\) with respect to the Segre embedding. We denote the Segre product by \(X \segre Y \defeq \sigma(X \times Y)\).

In what follows, rather than the geometric setting described above, we will be interested in the dual notion of the Segre product of graded algebras. Let \(A=\bigdsum_{i\in \nat} A_{i} \) and \(B=\bigdsum_{i\in \nat} B_{i} \) be \(\nat\)-graded $\bK$-algebras. Then their \textup{Segre product} \(A \segre B\) is the \(\nat\)-graded algebra
    \begin{equation}\label{eq:segre-def} A \segre B \defeq \bigdsum_{i \in \nat} A_i \tensor_{\bK} B_i\end{equation}
    with the usual tensor product algebra multiplication. Letting \(X\) and \(Y\) be projective varieties with homogeneous coordinate rings \(A\) and \(B\) respectively, the Segre product \(A \segre B\) is the homogeneous coordinate ring of \(X \segre Y\). 

Cluster algebras are a class of combinatorially rich algebras arising in a number of algebraic and geometric contexts (see \cite{FWZ1} and references therein).  The additional data of a cluster structure leads to the existence of canonical bases, closely related to the canonical bases arising in Lie theory.  Important examples of cluster algebras of this type include coordinate algebras of projective varieties and their various types of cells, e.g. Grassmannians (\cite{Scott-Grassmannians}), Schubert cells (\cite{GLS-KacMoody}) and positroid varieties (\cite{GL}).

In all known examples when the cluster algebra is the coordinate algebra of a projective variety, we have a compatible grading on the cluster algebra, with all cluster variables being homogeneous.  Such cluster algebras are naturally called \emph{graded cluster algebras} and the general theory of these is set out in work of the first author (\cite{GradedCAs}).

In this note, inspired by \cite[Remark 4.14]{P}, we define a cluster algebra structure on the Segre product of graded cluster algebras.  This generalises the particular case arising in \cite{P} in the study of cluster algebra structures on positroid varieties and in doing so, we are able to clarify the required input data to be able to form a Segre product.  

We show that from the point of view of cluster algebras, forming the Segre product is given by a gluing operation on suitable frozen variables.  We also record some simple observations on the preservation or otherwise of cluster-algebraic properties under taking Segre products.

%As we will see, the standard Segre construction imposes significant restrictions on both the graded cluster algebras and the choice of clusters at which one can glue.  The latter is perhaps surprising since most cluster algebra constructions are agnostic as to choices of initial cluster.  

\subsection*{Acknowledgements}

The authors acknowledge financial support from the Engineering and Physical Sciences Research Council (studentship ref.\ 2436773,\linebreak \url{https://gtr.ukri.org/projects?ref=studentship-2436773}).  We are also very grateful to the anonymous referee for many helpful suggestions that have greatly improved the exposition of our results.

\section{Segre Products of Graded Cluster Algebras}

It was shown by Galashin and Lam in \cite{GL} that coordinate rings of positroid varieties in the Grassmannian have cluster algebra structures. This class is closed under Segre product and in \cite{P}, Pressland shows how the Galashin--Lam cluster structure on the product is related to that on the factors.  

In what follows, we aim to generalise this construction to the case of graded skew-symmetric cluster algebras: we start with two graded cluster algebras and show that their Segre product has a natural cluster structure.  For coordinate rings of positroid varieties, Pressland's result shows that the Galashin--Lam cluster structure on the product is equal to that obtained by the Segre product construction we give here.

We start by establishing some notation; readers unfamiliar with graded cluster algebras may wish to refer to \cite{GradedCAs} for further details and examples.

First, let \(\cA_i = (\cluster{i}, \exch{i}, B_i, G_i)\) be (skew-symmetric) graded cluster algebras, for \(i\in \{1,2\}\), such that
\begin{itemize}
    \item \(\cluster{1} = \{ x_1, \dotsc, x_{n_1} \}\) and \(\cluster{2} = \{y_1, \dotsc, y_{n_2}\}\) are the respective initial clusters;
    \item \(\exch{i} \subsetneq \cluster{i}\) is the set of mutable cluster variables;
    \item every frozen variable (i.e.\ those elements in $\cluster{i}\setminus \exch{i}$) is invertible;
    \item \(B_i\) is an exchange matrix (with rows indexed by $\cluster{i}$ and columns by $\exch{i}$)  with skew-symmetric principal part;
    \item \(G_i \in \mathbb{Z}^{n_{i}}\) is a grading vector, i.e.\ a vector such that \(B_i^TG_i=0\).
\end{itemize}

It is common to visualise cluster mutation using quivers; see e.g. \cite{Marsh-book}.  We will do the same: an exchange matrix $B_{i}$ will be represented by an ice quiver having vertices $\bullet$ labelled by the elements of $\cluster{i}$. The frozen variables indicated by a box $\Box$, to indicate that mutation is not carried out there.  Arrows are determined by $B_{i}$: the number of arrows from $x_{j}$ to $x_{k}$ is $(B_{i})_{jk}$.

Throughout, we will work over a field $\bK$, so that our cluster algebras are $\bK$-algebras and we take all tensor products to be over $\bK$.  As we will see, the underlying field plays essentially no role in our construction.

Let $\cluster{}$ be a cluster with \(x\) a cluster variable and $B$ the exchange matrix associated to $\cluster{}$. We denote by \(B^{x}\) the row of \(B\) indexed by $x$ and by \(\hat{B}^{x}\) the matrix obtained from \(B\) by removing the row \(B^{x}\).  If $x$ is frozen, \(\hat{B}^{x}\) is again an exchange matrix.

\begin{remark}
In the above we require at least one frozen cluster variable in each cluster algebra---this will be important when defining a cluster structure on their Segre product since this will involve `gluing' at frozen variables.

We have also asked that every frozen variable is invertible, which is a common but not universal assumption in cluster theory.  In fact, an examination of our construction shows that this assumption can be weakened to only asking that the glued frozen variables are invertible, which may be a more appropriate assumption for geometric applications.
\end{remark}

We wish to define a cluster algebra structure on the Segre product \(\SegreCA{1}{2}\). Following the approach of \cite{P}, we aim to construct a new cluster algebra from \(\cA_1\) and \(\cA_2\) by gluing at frozen variables of the same degree, which we will show coincides with the Segre product under suitable further conditions. 

\subsection{A gluing construction}

Fix \(x \in \cluster{1}\setminus \exch{1}\) and \(y \in \cluster{2}\setminus \exch{2}\) such that $(G_{1})_{x}=(G_{2})_{y}$.  That is, $x$ and $y$ are frozen variables in their respective clusters and their degrees are equal.  We will identify the frozen variables \(x\) and \(y\), denoting a new proxy variable replacing both of these by \(z\). 

The initial data for our new cluster algebra is as follows. For the initial cluster, we take
\[\cluster{1}\Box \cluster{2} \defeq (\cluster{1}\setminus\{x\}) \union (\cluster{2}\setminus\{y\}) \union \{z\}.\]
The mutable variables are \(\exch{1} \union \exch{2}\), and for the initial exchange matrix, we form the block matrix
\[ B_{1}\Box B_{2} \defeq \begin{bmatrix} \hat{B}_1^{x} & 0 \\
0 & \hat{B}_2^{y} \\
B_{1}^{x} & B_{2}^{y}   
\end{bmatrix}.\]
Finally, for the initial grading vector we take
\[G_{1}\Box G_{2} \defeq \begin{bmatrix} \hat{G}_{1}^{x} \\ \hat{G}_{2}^{y} \\ G_{1}^{z} \end{bmatrix} \]
where \(\hat{G}_{1}^{x}\) is the grading vector \(G_{1}\) with the entry indexed by \(x\) removed (and similarly for $\hat{G}_{2}^{y}$) and \(G_{1}^{z} \defeq (G_1)_{x}=(G_{2})_{y}\). We can now define a cluster algebra \[ \cA_{1} \Box \cA_{2} = \cA(\cluster{1}\Box \cluster{2},\exch{1}\union \exch{2},B_{1}\Box B_{2},G_{1}\Box G_{2})\] from this initial data.

Let us extend the above notation to write
\[ \cluster{1}'\Box \cluster{2}' = (\cluster{1}'\setminus\{x\}) \union (\cluster{2}'\setminus\{y\}) \union \{z\}, \]
where $\cluster{1}'$, $\cluster{2}'$ are now allowed to be any clusters of $\cA_{1}$ and $\cA_{2}$, respectively, and say that $\cluster{1}'\Box \cluster{2}'$ is obtained by gluing $x$ and $y$. This is well-defined since $x$ and $y$ are frozen. Similarly, we extend the notation $B_{1}\Box B_{2}$ and $G_{1}\Box G_{2}$ to any appropriate input matrices/vectors.

The process of gluing at frozen variables with matching degree is illustrated in the example below. Here and elsewhere, $\mathds{1}$ denotes the vector $(1,\dotsc ,1)^{T}$.

\begin{example}
Let \(\cA_1 = (\cluster{1}= \{x_1,x_2,x_3\}, \exch{1}=\{x_{1}\}, Q_1, G_1=\mathds{1})\) and \(\cA_2 = (\cluster{2} = \{y_1,y_2,y_3\}, \exch{2}=\{y_{1}\}, Q_2, G_1=\mathds{1})\) be cluster algebras with exchange quivers as follows:
\begin{center}\begin{tikzpicture}
    \node[rectangle, draw,label=below:{$x_2$}] (11) {};

    \node[] (q1) [left of=11, xshift=25] {$Q_1:$};
    
    \node[fill=black, draw,circle, inner sep=2pt, label=below:{$x_1$}] (12) [right of=11] {};
    \node[rectangle, draw, label=below:{$x_3$}] (13) [right of=12] {};
    \node[rectangle, draw, label=below:{$y_3$}] (21) [right of=13, xshift=12] {};

    \node[] (q2) [left of=21, xshift=25] {$Q_2:$};
    
    \node[fill=black, draw,circle, inner sep=2pt,label=below:{$y_1$}] (22) [right of=21] {};
    \node[rectangle, draw, label=below:{$y_2$}] (23) [right of=22] {};

    \draw[-latex] (11) to node {} (12);
    \draw[-latex] (12) to node {} (13);
    \draw[-latex] (21) to node {} (22);
    \draw[-latex] (22) to node {} (23);
\end{tikzpicture}\end{center}
The quiver obtained by `gluing' at the frozen variables \(x_3\) and \(y_3\) is shown below---we denote the new variable by \(z\). 
\begin{center}\begin{tikzpicture}
    \node[rectangle, draw, label=below:{$x_2$}] (11) {};

    \node[] (q1) [left of=11, xshift=25] {$Q:$};
    
    \node[fill=black, draw,circle, inner sep=2pt, label=below:{$x_1$}] (12) [right of=11] {};
    \node[rectangle, draw, label=below:{$z$}] (13) [right of=12] {};
    
    \node[fill=black, draw,circle, inner sep=2pt, label=below:{$y_1$}] (22) [right of=13] {};
    \node[rectangle, draw, label=below:{$y_2$}] (23) [right of=22] {};

    \draw[-latex] (11) to node {} (12);
    \draw[-latex] (12) to node {} (13);
    \draw[-latex] (13) to node {} (22);
    \draw[-latex] (22) to node {} (23);
\end{tikzpicture}\end{center}
The cluster algebra \(\cA_{1} \Box \cA_{2}\) is then given by the initial data \[ (\cluster{} = \{x_1, x_2, y_1, y_2, z\},\exch{} = \{x_{1},y_{1}\}, Q, G=\mathds{1}). \] We will show in Theorem \ref{t:iso-to-Segre} that this gives a cluster structure on the Segre product \(\SegreCA{1}{2}\).
\end{example}

We record some straightforward observations about the cluster algebra $\cA_{1} \Box \cA_{2}$.

\begin{lemma}\label{l:box-prod-comm}
    Let $\cA_{1}$ and $\cA_{2}$ be graded cluster algebras.  Fix \(x \in \cluster{1}\setminus \exch{1}\) and \(y \in \cluster{2}\setminus \exch{2}\) such that $(G_{1})_{x}=(G_{2})_{y}$.  Then the cluster algebras $\cA_{1}\Box \cA_{2}$ and $\cA_{2}\Box \cA_{1}$ are isomorphic as cluster algebras.
\end{lemma}

\begin{proof}
    This is clear from comparing the initial data for $\cA_{1}\Box \cA_{2}$ and $\cA_{2}\Box \cA_{1}$ and in particular noting that the two initial clusters are equal up to permutation of the entries.
\end{proof}

\begin{lemma}\label{l:box-prod-var-cl-corresp}
    Let $\cA_{1}$ and $\cA_{2}$ be graded cluster algebras. Fix \(x \in \cluster{1}\setminus \exch{1}\) and \(y \in \cluster{2}\setminus \exch{2}\) such that $(G_{1})_{x}=(G_{2})_{y}$. 

    \begin{enumerate}[label=\textup{(\roman*)}]
        \item\label{l:box-prod-var-corresp} Every cluster variable of $\cA_{1}\Box \cA_{2}$ is naturally identified with a cluster variable of $\cA_{1}$, a cluster variable of $\cA_{2}$ or is equal to $z$.
        \item\label{l:box-prod-cl-corresp} There is a bijection between pairs of clusters $(\cluster{1}',\cluster{2}')$ and clusters of $\cA_{1}\Box \cA_{2}$ given by gluing, i.e.\ sending $(\cluster{1}',\cluster{2}')$ to $\cluster{1}'\Box \cluster{2}'$ for a cluster $\cluster{1}'$ of $\cA_{1}$ and $\cluster{2}'$ of $\cA_{2}$.
    \end{enumerate} 
\end{lemma}

\begin{proof}
    This follows from observing that our gluing process does not introduce any new arrows between mutable vertices.  Since mutation is a local phenomenon and concentrated on mutable vertices, it is straightforward to see that mutating at vertices indexed by $\exch{1}$ is independent of mutating at vertices indexed by $\exch{2}$ and the (mutable) variables obtained are exactly as if the gluing had not been carried out. The frozen variables of $\cA_{1}\Box \cA_{2}$ are those of $\cA_{1}$ and $\cA_{2}$ excluding $x$ and $y$, along with the glued frozen $z$.
    
    For the second part, note that the same argument shows that there is a similar bijection for the clusters of $\cA_{1}\times \cA_{2}$, where the latter denotes the ``disconnected'' product of cluster algebras, where one simply takes the union of clusters and direct sum of exchange matrices.  Now there is evidently a bijection between the clusters of $\cA_{1}\times \cA_{2}$ and those of $\cA_{1}\Box \cA_{2}$, given by $\cluster{1}'\union \cluster{2}' \mapsto \cluster{1}'\Box \cluster{2}'$, from which the claim follows.
\end{proof}

\begin{corollary}\label{c:box-prod-fin-type}
	Let $\cA_{1}$ and $\cA_{2}$ be graded cluster algebras. Fix \(x \in \cluster{1}\setminus \exch{1}\) and \(y \in \cluster{2}\setminus \exch{2}\) such that $(G_{1})_{x}=(G_{2})_{y}$. 
	
	Then 
	\begin{enumerate}[label=\textup{(\roman*)}]
		\item $\cA_{1}\Box \cA_{2}$ is of finite type if and only if $\cA_{1}$ and $\cA_{2}$ are;
		\item writing $\kappa(\cA)$ for the number of cluster variables of a cluster algebra $\cA$, we have $\kappa(\cA_{1}\Box \cA_{2})=\kappa(\cA_{1})+\kappa(\cA_{2})-1$ when these numbers are all finite; and
		\item writing $K(\cA)$ for the number of clusters of $\cA$, we have $K(\cA_{1}\Box \cA_{2})=K(\cA_{1})K(\cA_{2})$ when these numbers are all finite.
	\end{enumerate} 
\end{corollary}

\begin{proof} These are now immediate from the previous lemma.  Note that there is an overall reduction of one in the number of cluster variables because we have glued two previously distinct frozen variables; this highlights the difference between this construction and the disconnected product.
\end{proof}

\begin{remark}
    One might hope that this construction extends straightforwardly to graded quantum cluster algebras (cf.\ \cite{GradedQCAs}).  However, computation in small examples shows that this is not the case.  
    
    For if one tries the na\"{\i}ve approach in which initial quantum cluster variables from $\cA_{1}$ \emph{commute} with those from $\cA_{2}$, one rapidly finds situations in which after performing a mutation, the new variable does not quasi-commute with the rest of its cluster.  For it to do so requires the compatibility condition between the exchange and quasi-commutation matrices for the glued data and this imposes a collection of ``cross-term'' requirements between $B_{1}$ and $L_{2}$ (respectively, $B_{2}$ and $L_{1}$) in respect of the glued frozen variables.
\end{remark}

\subsection{Relationship with the Segre product}

Our main result is the following theorem, showing that the cluster algebra construction $\cA_{1}\Box \cA_{2}$ induces a cluster algebra structure on the Segre product.  The isomorphism we will use is directly analogous to the map \(\delta^{\text{src}}\) defined in \cite{P}.

\begin{theorem}\label{t:iso-to-Segre}
    Let $\cA_{i}=(\cluster{i},\exch{i},B_{i},G_{i})$, $i=1,2$ be graded cluster algebras such that there exist \(x \in \cluster{1}\setminus \exch{1}\) and \(y \in \cluster{2}\setminus \exch{2}\) both of degree 1.
    
    Then the map \(\varphi: \cA_{1}\Box \cA_{2} \to \SegreCA{1}{2}\) given on initial cluster variables by
\begin{align*}
    \varphi(x_{j}) &= x_j \tensor y^{\deg x_{j}} \quad \text{for } x_{j} \in \cluster{1}\setminus \{x\}, \\
    \varphi(y_{j}) &= x^{\deg y_{j}} \tensor y_j \quad \text{for } y_{j} \in \cluster{2}\setminus \{y\}\ \text{and}\\
    \varphi(z) &= x \tensor y
\end{align*}
    is a graded algebra isomorphism, with the property that the above formul\ae\ hold for any cluster of $\cA_{1}\Box \cA_{2}$.
\end{theorem}

\begin{proof}
Recalling that we set

\[\cluster{1}\Box \cluster{2} = (\cluster{1}\setminus\{x\}) \union (\cluster{2}\setminus\{y\}) \union \{z\},\] let $\varphi$ denote the algebra homomorphism $\varphi\colon \bK(\cluster{1}\Box \cluster{2})\to \bK(\cluster{1})\tensor \bK(\cluster{2})$
    obtained from the above specification on generators of the domain.  This map is injective since the elements $\varphi(x_{j})$, $\varphi(y_{j})$ and $\varphi(z)$ are algebraically independent.
    
    Now let $\varphi$ denote the restriction of the above map to $\cA_{1}\Box \cA_{2}$.  We first claim that the restricted map \(\varphi\) takes values in the subalgebra $\cA_{1}\tensor \cA_{2}$. To prove this, we proceed by induction on the number of mutation steps from the initial cluster. 
    
    We may take as base case that of zero mutations from the initial cluster: there is nothing to do, since we see immediately that $\varphi(x_{j})$, $\varphi(y_{j})$ and $\varphi(z)$ lie in $\cA_{1}\tensor \cA_{2}$ by definition.
    
    Now assume that the result holds $r-1$ mutations from the initial cluster $\cluster{}=\cluster{1}\Box \cluster{2}$ (for $r\geq 1$) of $\cA_{1}\Box \cA_{2}$.  That is, let $\underline{y}=\mu_{k_{r-1}}\mu_{k_{r-2}}\dotsm \mu_{k_{1}}(\cluster{})$.  Set $B=\mu_{k_{r-1}}\mu_{k_{r-2}}\dotsm \mu_{k_{1}}(B_{1}\Box B_{2})$.
    
    By Lemma~\ref{l:box-prod-var-cl-corresp}\ref{l:box-prod-cl-corresp}, we have that $\underline{y}=\underline{y}_{1}\Box \underline{y}_{2}$ for some clusters $\underline{y}_{1}$, $\underline{y}_{2}$ of $\cA_{1}$ and $\cA_{2}$ respectively.  Moreover, there is a decompostion 
    \[ \{k_{1},\dotsc,k_{r-1} \}=\{ l_{1},\dotsc ,l_{s}\}\sqcup \{ m_{1},\dotsc, m_{t}\}\]
    such that $\underline{y}_{1}=\mu_{l_{s}}\dotsm \mu_{l_{1}}(\cluster{1})$ and $\underline{y}_{2}=\mu_{m_{t}}\dotsm \mu_{m_{1}}(\cluster{2})$.
    
    Let  \(\underline{y}_{1} = \{ x_1, \dotsc, x_{n_1} \}\) and \(\underline{y}_{2} = \{y_1, \dotsc, y_{n_2}\}\), so that
    \[ \underline{y}=\underline{y}_{1}\Box \underline{y}_{2}=(\{ x_1, \dotsc, x_{n_1} \}\setminus \{ x\}) \sqcup (\{y_1, \dotsc, y_{n_2}\}\setminus \{ y \}) \sqcup \{ z \} \]
    
    Let $C=\mu_{l_{s}}\dotsm \mu_{l_{1}}(B_{1})$, $D=\mu_{m_{t}}\dotsm \mu_{m_{1}}(B_{2})$, $H=\mu_{l_{s}}\dotsm \mu_{l_{1}}(G_{1})$ and $K=\mu_{m_{t}}\dotsm \mu_{m_{1}}(G_{2})$.  Then in particular $B=C\Box D$ and $H_{j}=\deg x_{j}$ and $K_{j}=\deg y_{j}$. We also set $[n]_{+}=\max \{ n,0\}$ and $[n]_{-}=\max \{ -n,0 \}$.
    
    We then compute \(\varphi\) for one further mutation in direction $k_{r}=k$. We first consider the case in which \(x_k \in \underline{y}_{1}\) is mutable. 
    
    %\JGeditcomment{GTH 21/8}
We have
    \begin{align*}
        \varphi( \mu_k(x_k)) &= \varphi \Bigg( \frac{1}{x_k} \Bigg[  \Bigg(\prod\limits_{x_{j}\in \underline{y}_{1}\setminus\{x\}} x_j^{[B_{x_{j},x_{k}}]_{+}}\Bigg) \Bigg(\prod\limits_{y_{j}\in \underline{y}_{2}\setminus\{y\}} y_j^{[B_{y_{j},x_{k}}]_{+}} \Bigg)z^{[B_{z,x_{k}}]_+}  \\ 
        &\qquad \qquad+   \Bigg(\prod\limits_{x_{j}\in \underline{y}_{1}\setminus \{x\}} x_j^{[B_{x_{j},x_{k}}]_{-}}\Bigg) \Bigg(\prod\limits_{y_{j}\in \underline{y}_{2}\setminus \{y\}} y_j^{[B_{y_{j},x_{k}}]_{-}} \Bigg)z^{[B_{z,x_{k}}]_-} \Bigg] \Bigg) \\
        &= \varphi \Bigg( \frac{1}{x_k} \Bigg[  \Bigg(\prod\limits_{x_{j}\in \underline{y}_{1}\setminus\{x\}} x_j^{[B_{x_{j},x_{k}}]_{+}}\Bigg) z^{[B_{z,x_{k}}]_+}  +   \Bigg(\prod\limits_{x_{j}\in \underline{y}_{1}\setminus \{x\}} x_j^{[B_{x_{j},x_{k}}]_{-}}\Bigg) z^{[B_{z,x_{k}}]_-} \Bigg] \Bigg) \\
        &= \varphi \Bigg( \frac{1}{x_k} \Bigg[  \Bigg(\prod\limits_{x_{j}\in \underline{y}_{1}\setminus\{x\}} x_j^{[C_{x_{j},x_{k}}]_{+}}\Bigg) z^{[C_{x,x_{k}}]_+}  +   \Bigg(\prod\limits_{x_{j}\in \underline{y}_{1}\setminus \{x\}} x_j^{[C_{x_{j},x_{k}}]_{-}}\Bigg) z^{[C_{x,x_{k}}]_-} \Bigg] \Bigg) \\
        &= \frac{1}{x_k \otimes y^{\deg x_{k}}} \Bigg[ \prod\limits_{x_{j}\in \underline{y}_{1}\setminus \{x\}}\Bigg( x_j^{[C_{x_{j},x_{k}}]_{+}} \otimes y^{[C_{x_{j},x_{k}}]_{+}\deg x_{j}} \Bigg)x^{[C_{x,x_{k}}]_+}\tensor y^{[C_{x,x_{k}}]_+} \\ 
        & \phantom{=} + \prod\limits_{x_{j}\in \underline{y}_{1}\setminus \{x\}}\Bigg( x_j^{[C_{x_{j},x_{k}}]_{-}} \otimes y^{[C_{x_{j},x_{k}}]_{-}\deg x_{j}} \Bigg)x^{[C_{x,x_{k}}]_-}\tensor y^{[C_{x,x_{k}}]_-} \Bigg] \\
        &= \frac{1}{x_k \otimes y^{\deg x_{k}}} \Bigg[ \prod\limits_{x_{j}\in \underline{y}_{1}} x_j^{[C_{x_{j},x_{k}}]_{+}} \otimes y^{d} + \prod\limits_{x_{j}\in \underline{y}_{1}} x_j^{[C_{x_{j},x_{k}}]_{-}}\otimes y^{d}  \Bigg] \\
        &= \frac{1}{x_k}\Bigg(\prod\limits_{x_{j}\in \underline{y}_{1}} x_j^{[C_{x_{j},x_{k}}]_{+}}+ \prod\limits_{x_{j}\in \underline{y}_{1}} x_j^{[C_{x_{j},x_{k}}]_{-}}\Bigg) \otimes y^{d-\deg x_{k}} \\
        &= \mu_k(x_k) \otimes y^{d-\deg x_{k}} \\
        &= \mu_k(x_k) \otimes y^{\deg \mu_k(x_k)}
    \end{align*}
    where \begin{align*} d & = \sum_{x_{j}} [C_{x_{j},x_{k}}]_{+}\deg x_{j} = \sum_{C_{x_{j},x_{k}}>0} C_{x_{j},x_{k}}H_{x_{j}} \\ & = \sum_{C_{x_{j},x_{k}}<0} -C_{x_{j},x_{k}}H_{x_{j}} = \sum_{x_{j}} [C_{x_{j},x_{k}}]_{-}\deg x_{j} \end{align*} noting that the third equality holds since $C^{T}H=0$.  Also, we use that $\deg \mu_{k}(x_{k})=d-\deg x_{k}$. 

	Note that the fifth equality is where the assumption that $\deg x=1$ is used: without it, the claimed equality of $d$ with the other stated quantities need not hold.

    An analogous argument shows that \(\varphi(\mu_k(y_k)) = x^{\deg \mu_k(y_k)} \otimes \mu_k(y_k)\) for \(y_k \in \underline{y}_{2}\) mutable, noting that this time, it is $\deg y=1$ that is required.
    %Now, let \(\underline{z} = (z_1, \dotsc, z_{n_1+n_2-1})\) be a cluster \(m\) mutation steps away from the initial cluster \(\underline{x}\), i.e. \(\underline{z} = \mu_{\underline{p}}(\underline{x})\) for some mutation path \(\underline{p}\) of length \(m\), and assume that
    %\[\varphi(z_j) = z_j \otimes y_{s_2}^{\deg z_j}\]
    %for \(z_j \in \cA_1\). Denote by \(B'\) the corresponding exchange matrix. We claim that, for \(z_k \in \cA_1\), 
    %\[\varphi(\mu_k(z_k)) = \mu_k(z_k) \otimes y_{s_2}^{\deg \mu_k(z_k)}.\]
%    Indeed, the same calculation as above with $\cluster[z]{}$ and $B'$ shows that this is the case and similarly \(\varphi(\mu_k(z_k)) = x_{s_1}^{\deg \mu_k(z_k)}\otimes \mu_k(z_k)\), when \(z_k \in \cA_2\).
     %The above allows us to define \(\varphi\colon \cA_{1}\Box \cA_{2} \to \cA_{1}\tensor \cA_{2}\) on the generating set of cluster variables as above.  We have seen that this respects the defining (exchange) relations and can therefore be extended to an algebra homomorphism.  It is clearly injective on the generating set of the domain and therefore injective.
     
     Since we have $\deg x=\deg y=1$, the above tells us that for any cluster variable $x'$ of $\cA_{1}\Box \cA_{2}$, we either have $\varphi(x')=x' \tensor y^{\deg x'}$ or $\varphi(x')=x^{\deg x'} \tensor x'$ and hence $\varphi(x')\in (\cA_{1})_{\deg x'} \tensor (\cA_{2})_{\deg x'}$. 
     %    \begin{align*}
     	%    \varphi(x_{j}) &= x_j \tensor y^{\deg x_{j}} \in (\cA_{1})_{\deg x_{j}}\tensor (\cA_{2})_{\deg x_{j}} \\
     	%    \varphi(y_{j}) &= x^{\deg y_{j}} \tensor y_j \in (\cA_{1})_{\deg y_{j}} \tensor (\cA_{2})_{\deg y_{j}}\\
     	%    \varphi(z) &= x \tensor y \in (\cA_{1})_{1} \tensor (\cA_{2})_{1}
     	%\end{align*}
     	That is, the image of $\varphi$ is contained in the Segre product $\SegreCA{1}{2}$ without any further constraints and the map $\varphi$ is a graded map.
     	
	It remains to check surjectivity.  Note that a generating set for $\SegreCA{1}{2}$ is given by taking the elementary tensors with components in generating sets for $\cA_{1}$ and $\cA_{2}$, i.e.\ 
	\[\{z_1 \tensor z_2 | z_1 \in (\cA_1)_d, z_2 \in (\cA_2)_d\ \text{cluster variables}, d\in \integ \}\]     
	
	Now
	\begin{align*}
		z_1 \tensor z_2 &= (z_1 \tensor y^{d})(x^{d} \tensor z_2)(x^{-d} \tensor y^{-d}) = \varphi(z_1)\varphi(z_2)\varphi(z)^{-d}.
	\end{align*}
	Hence, \(\Image \varphi\) contains a generating set for \(\SegreCA{1}{2}\), and so \(\varphi\) is surjective onto $\SegreCA{1}{2}$. The claim follows.     	
\end{proof}

\begin{remark} One might be tempted to try changing the specification of the map $\varphi$ to
	\begin{align*}
	\varphi(x_{j}) &= x_j^{\deg y} \tensor y^{\deg x_{j}} \quad \text{for } x_{j} \in \cluster{1}\setminus \{x\}, \\
	\varphi(y_{j}) &= x^{\deg y_{j}} \tensor y_j^{\deg x} \quad \text{for } y_{j} \in \cluster{2}\setminus \{y\}\ \text{and}\\
	\varphi(z) &= x^{\deg y} \tensor y^{\deg x}
\end{align*}	
	in an attempt to avoid the $\deg x=\deg y=1$ assumption.  Note that one should however ask for $\deg x$ and $\deg y$ strictly positive, to avoid issues with needing inverses of arbitrary cluster variables.
	
	While this does indeed fix the issue with $d$ that occurs in the calculation in the above proof for $x_{k}\in \underline{y}_{1}$, the appearance of $x^{\deg y}$ in the first tensor factor means that we do not obtain $\mu_{k}(x_{k})$ unless $\deg y=1$.
	
	More explicitly, following the same approach as in the previous proof, one would arrive at
	\[ \frac{1}{x_k^{\deg y} \otimes y^{\deg x_{k}}} \Bigg[ \prod\limits_{x_{j}\in \underline{y}_{1}} x_j^{[C_{x_{j},x_{k}}]_{+}\deg y} \otimes y^{d} + \prod\limits_{x_{j}\in \underline{y}_{1}} x_j^{[C_{x_{j},x_{k}}]_{-}\deg y}\otimes y^{d}  \Bigg] \] but this is not equal to 
	\[ \frac{1}{x_k^{\deg y}}\Bigg(\prod\limits_{x_{j}\in \underline{y}_{1}} x_j^{[C_{x_{j},x_{k}}]_{+}}+ \prod\limits_{x_{j}\in \underline{y}_{1}} x_j^{[C_{x_{j},x_{k}}]_{-}}\Bigg)^{\deg y} \otimes y^{d-\deg x_{k}} \]
	if $\deg y\neq 1$.
	
	By symmetry, the other case tells us that we also need $\deg x=1$.  That is, the degree 1 assumption is unavoidable.	
\end{remark}

\begin{remark}
    Notice that in proving surjectivity, we required $\varphi(z)=x\tensor y$, and hence $x$ and $y$ themselves, to be invertible, but no other frozen variables needed to be invertible for the proof to hold.
\end{remark}

%\begin{corollary}
%    Let $\cA_{i}=(\cluster{i},\exch{i},B_{i},G_{i}=\mathds{1})$, $i=1,2$ be graded cluster algebras such that $\cluster{1}$ and $\cluster{2}$ are homogeneous of degree 1 and fix \(s_1 \in \oneton{n_{1}}\setminus \exch{1}\) and \(s_2 \in \oneton{n_{2}}\setminus \exch{2}\).
%
%    Then 
%    \begin{enumerate}[label=\textup{(\roman*)}]
%        \item the cluster algebras $\SegreCA{1}{2}$ and $\SegreCA{2}{1}$ are isomorphic as cluster algebras;
%        \item $\SegreCA{1}{2}$ is of finite type if and only if $\cA_{1}$ and $\cA_{2}$ are;
%        \item writing $\kappa(\cA)$ for the number of cluster variables of a cluster algebra $\cA$, we have $\kappa(\SegreCA{1}{2})=\kappa(\cA_{1})+\kappa(\cA_{2})-1$ when these numbers are all finite; and
%        \item writing $K(\cA)$ for the number of clusters of $\cA$, we have $K(\SegreCA{1}{2})=K(\cA_{1})K(\cA_{2})$ when these numbers are all finite. \qed
%    \end{enumerate} 
%\end{corollary}

\begin{remark}
	Via Lemma~\ref{l:box-prod-var-cl-corresp}, we see that the cluster structure on $\cA_{1}\Box \cA_{2}$ and hence that on $\SegreCA{1}{2}$ is independent of the choices of initial seeds.  Therefore the only requirements to obtain a cluster structure on the Segre product are the existence of a frozen variable of degree 1 for each factor.
	
	Graded cluster algebras with at least one frozen variable of degree one are, perhaps surprisingly, ubiquitous.  Many examples arising geometrically have this property: coordinate rings of Grassmannians and more generally partial flag varieties and their cells (\cite{GLS-KacMoody}), double Bruhat cells (\cite{BFZ-CA3}) and, as motivated this work, positroid varieties (\cite{GL}).
%    In the above theorem, we require the input clusters to be homogeneous of degree one. This was necessary to ensure that the image of \(\varphi\) generates the Segre product as defined. We note that this condition is very restrictive and it would be desirable for it to be weakened.  However, without it, it does not seem feasible to describe the image in as simple a fashion as for the standard Segre product.
    
%    We also observe that, as a result, the construction of $\cA_{1}\Box \cA_{2}$ and hence the cluster structure induced on $\SegreCA{1}{2}$ is strongly ``rooted'', i.e.\ dependent on a choice of initial clusters with specific properties.  Even in a graded cluster algebra that has a homogeneous cluster, only with very specific exchange matrices we will find that other clusters are also homogeneous.  
\end{remark}

Note too that the claims on the cluster structure of $\cA_{1}\Box \cA_{2}$ in Corollary~\ref{c:box-prod-fin-type} therefore also apply to the induced cluster structure on the Segre product. 

\bibliographystyle{halpha}
\bibliography{biblio}

\begin{thebibliography}{FWZ21}

\bibitem[BFZ05]{BFZ-CA3}
Arkady Berenstein, Sergey Fomin, and Andrei Zelevinsky.
\newblock Cluster algebras. {III}. {U}pper bounds and double {B}ruhat cells.
\newblock {\em Duke Math. J.}, 126(1):1--52, 2005, arXiv:math/0305434.

\bibitem[FWZ21]{FWZ1}
S.~Fomin, L.~Williams, and A~Zelevinsky.
\newblock Introduction to {C}luster {A}lgebras. {C}hapters 1-3, 2021,
  arXiv:1608.05735.

\bibitem[GL14]{GradedQCAs}
J.E. Grabowski and S.~Launois.
\newblock Graded quantum cluster algebras and an application to quantum
  {G}rassmannians.
\newblock {\em Proc. Lond. Math. Soc. (3)}, 109(3):697--732, 2014,
  arXiv:1301.2133.

\bibitem[GL19]{GL}
P.~Galashin and T.~Lam.
\newblock Positroid varieties and cluster algebras.
\newblock {\em Annales Scientifiques de l'{\'E}cole Normale Sup{\'e}rieure},
  2019, arXiv:1906.03501.

\bibitem[GLS11]{GLS-KacMoody}
C.~Gei{\ss}, B.~Leclerc, and J.~Schr{\"o}er.
\newblock Kac-{M}oody groups and cluster algebras.
\newblock {\em Adv. Math.}, 228(1):329--433, 2011, arXiv:1001.3545.

\bibitem[Gra15]{GradedCAs}
J.E. Grabowski.
\newblock Graded cluster algebras.
\newblock {\em J. Algebraic Combin.}, 42(4):1111--1134, 2015, arXiv:1309.6170.

\bibitem[Mar14]{Marsh-book}
Bethany~R. Marsh.
\newblock {\em Lecture notes on cluster algebras}.
\newblock Z\"{u}rich Lectures in Advanced Mathematics. European Mathematical
  Society, 2014.

\bibitem[Pre23]{P}
M.~Pressland.
\newblock Quasi-coincidence of cluster structures on positroid varieties, 2023,
  arXiv:2307.13369.

\bibitem[Sco06]{Scott-Grassmannians}
J.S. Scott.
\newblock Grassmannians and cluster algebras.
\newblock {\em Proc. London Math. Soc. (3)}, 92(2):345--380, 2006.

\end{thebibliography}

\end{document}